\documentclass[12pt, a4paper]{amsart}
\usepackage{amsmath, amssymb,amsthm}
\usepackage{lettrine}
\usepackage[Sonny]{fncychap}
\usepackage[english]{babel}
\usepackage{amsmath, amssymb,amsthm}
\usepackage[all]{xy}
\usepackage{tabularx}
\usepackage[applemac]{inputenc}
\usepackage{graphicx} 
\usepackage[dvipsnames]{xcolor}
\newtheorem{theorem}{Theorem}[section]
 
\newtheorem{prop}[theorem]{Proposition}
\newtheorem{lemma}[theorem]{Lemma}

\theoremstyle{definition}
\newtheorem{rem}{Remark}

\newcommand{\C}{\mathbf{C}}
\newcommand{\Q}{\mathbf{Q}}

\newcommand{\Z}{\mathbf{Z}}

\newcommand{\cO}{\mathcal{O}}

\newcommand{\cT}{\mathcal{T}}

\DeclareMathOperator{\Frob}{Frob}

\DeclareMathOperator{\SL}{SL}
\DeclareMathOperator{\Gal}{Gal}

\DeclareMathOperator{\PGL}{PGL}
\DeclareMathOperator{\rH}{H}

\DeclareMathOperator{\Spec}{Spec}

\title{On the $p$-adic periods of the modular curve $X(\Gamma_0(p) \cap \Gamma(2))$.}
\author{Adel Betina and Emmanuel Lecouturier}

\address{Universitat Polit\`ecnica de Catalunya} 
\email{adel.betina@upc.edu}  
\address{Universit\'e Paris 7}
\email{emmanuel.lecouturier@imj-prg.fr}  

\begin{document}

\maketitle
\begin{abstract}
We prove a variant of Oesterl\'e's conjecture describing $p$-adic periods of the modular curve $X_0(p)$, with an additional $\Gamma(2)$-structure (and also $\Gamma(3) \cap \Gamma_0(p)$ if $p \equiv 1 \text{ (mod } 3 \text{)}$). We use de Shalit's techniques and $p$-adic uniformization of curves with semi-stable reduction.

\end{abstract}

\section{Introduction}
Let $p \geq 5$ be a prime.
Let $X_0(p)$ be a regular, proper and flat model over $\Z_p$ of the modular curve of level $\Gamma_0(p)$, obtained by several blowing-up of the coarse moduli space of the algebraic stack $\mathfrak{M}_{\Gamma_0(p)}$ parametrizing generalised elliptic curves with $\Gamma_0(p)$-level structure \cite{D-R}. Deligne and Rapoport proved that the special fiber of $X_0(p)$ is semi-stable. Let $K$ be the quadratic unramified extension of $\Q_p$, $\mathcal{O}_K$ be the ring of integers of $K$, $k$ be the residue field of $\mathcal{O}_K$ and $X_0^{rig}(p)_K$ be the rigid variety associated to the proper flat curve $X_0(p)_K$. Since the supersingular points of the special fiber of $X_0(p)_{\mathbf{F}_p}$ are $k$-rational, as well as its irreducible components, Mumford's theorem \cite{M} implies that $X_0^{rig}(p)_K$ is the quotient of a $p$-adic half plane $\mathcal{H}_\Gamma=\mathbf{P}_K^1-\mathcal{L}$ by a Schottky group $\Gamma$, where $\mathcal{L}$ is the set of the limits points of $\Gamma$. Manin and Drinfeld constructed a pairing $\Phi: \Gamma^{ab} \times \Gamma^{ab} \rightarrow K^{\times}$ in \cite{D} and explained how this pairing gives a $p$-adic uniformization of the Jacobian $J_0(p)_K$ of $X_0(p)_K$.

Let $\Delta$ be the dual graph of the special fiber of $X_0(p)_{\cO_K}$. Mumford's construction shows that $\Gamma$ is isomorphic to the fundamental group $\pi_1(\Delta)$. The abelianization of $\Gamma$ is isomorphic to to the augmentation subgroup of the free $\mathbf{Z}$-module with basis the isomorphism classes of supersingular elliptic curves over $\overline{\mathbf{F}}_p$. Oesterl\'e conjectured that the pairing $\Phi$ can be expressed, modulo the principal units, in terms of the modular invariant $j$. E. de Shalit proved this conjecture in \cite{shalit} (up to a sign if $p \equiv 3 \text{ (mod }4\text{)}$). The aim of this paper is to prove a variant of Oesterl\'e conjecture for modular curves with a $\Gamma_0(p) \cap \Gamma(2)$-level structure by replacing $j$ by the modular invariant $\lambda$.

\section*{Notation}
\begin{enumerate}

\item For any algebraic stack $X$ over $\Spec A$ and for any morphism $A \rightarrow B$, we denote by $X_B$ the fiber product $X \times_{\Spec A} \Spec B$.
\item For any algebraic extension $k$ of the field $\Z/p\Z$, we denote by $W(k)$ the ring of Witt vectors associated to $k$ and by $\bar{k}$ the separable closure of $k$.
\item For any congruence subgroup $\Gamma$ of $\SL_2(\mathbf{Z})$, we denote by $\mathfrak{M}_\Gamma$ the stack over $\Z$ whose $S$-points classify generalized elliptic curves over $S$ with a $\Gamma$-level structure.
\item For any congruence subgroup $\Gamma$ of $\SL_2(\mathbf{Z})$ and any $c \in \mathbf{P}^1(\mathbf{Q})$, we denote by $[c]_{\Gamma}$ the cusp of $\Gamma \backslash (\mathcal{H} \cup \mathbf{P}^1(\mathbf{Q}))$ corresponding to the class of $c$, where $\mathcal{H}$ is the (complex) upper-half plane.

\item For two congruence subgroups $\Gamma$ and $L$, we denote by $\mathfrak{m}_{\Gamma \cap L}$ for the fiber product of algebraic stacks $\mathfrak{M}_{\Gamma} \times_{\mathfrak{M}} \mathfrak{M}_L$, where $\mathfrak{M}$ is the stack over $\Z$ whose $S$-points classify generalized elliptic curves over the scheme $S$.
\item For any algebraic stack $\mathfrak{M}_{\Gamma}$ over $\Z$, we denote by $M_{\Gamma}$ the coarse moduli space attached to $\mathfrak{M}_{\Gamma}$ ($M_{\Gamma}$ is an algebraic space). 

\item For any proper and flat scheme $\mathfrak{X}$ over $\mathcal{O}_K$, we denote by $\mathfrak{X}_K^{an}$ the rigid analytic space given by the generic fiber of the completion of $\mathfrak{X}$ along its special fiber (we have a Galois-equivariant isomorphism $\mathfrak{X}_K(\bar{K}) \simeq \mathfrak{X}^{an}_K(\bar{K}))$.

\end{enumerate}

We refer the reader to \cite[IV]{D-R} for more details about the stacks defined above.

Let $\mathfrak{X}$ be the proper flat normal coarse moduli space over $\Spec \mathcal{O}_K$ associated to the algebraic stack $\mathfrak{M}_{\Gamma_0(p) \cap \Gamma(2)}$. Deligne-Rapoport proved in \cite[VI.6.9]{D-R} that the singularities of $\mathfrak{X} $ are ordinary, and that the special fiber of $\mathfrak{X} $ is an union of two copies of $M_{\Gamma(2)}{ \small \otimes k}$ meeting transversally at the supersingular points, and such that a supersingular point $x$ of the first copy is identified with the point $x^p=\Frob_p(x)$ of the second copy (the supersingular points of the special fiber of $\mathfrak{X}$ are $k$-rational).

\begin{theorem}\label{main-thm}\

\begin{enumerate} 
\item The scheme $\mathfrak{X} $ is regular and the irreducible components of its special fiber are isomorphic to $\mathbf{P}^1_k$. 
\item Let $S:=\{e_i\}$ ($i \in \{0,...,g\}$ where $g = \frac{p-3}{2}$ is the genus of $\mathfrak{X}$) be the set of supersingular points of $\mathfrak{X}_{k}$ and let $\Gamma \subset \PGL_2(K)$ be a Schottky group such that there exists an isomorphism of rigid spaces $\mathfrak{H}_{\Gamma}/\Gamma \simeq \mathfrak{X}_K^{an}$, where $\mathfrak{H}_{\Gamma}=\mathbf{P}^1_K - \mathcal{L}$ and $\mathcal{L}$ are the limits points of $\Gamma$. Then the Drinfeld pairing $\Phi:\Gamma^{ab} \times \Gamma^{ab} \rightarrow K^{\times}$  takes values in $\Q_p^{\times}$.

\item Let $\bar{\Phi}$ be the residual pairing modulo the principal units $U_1(\Q_p^{\times})$ of $\Q_p^{\times}$, then, after the identification $\Gamma^{ab} \simeq H_1(\Delta, \mathbf{Z}) \simeq \Z[S]^0$ (the augmentation subgroup of $\mathbf{Z}[S]$), $\bar{\Phi}$ extends to a pairing (still denoted $\bar{\Phi}$) $\Z[S]\times \Z[S] \rightarrow K^{\times}/U_1(K)$ such that:

{\small 
 $$ \bar{\Phi}(e_i,e_j) = \left\{ 
 \begin{array}{cl}
 (\lambda(e_i)-\lambda(e_j))^{p+1} & \mbox{ if  }  i \ne j \text{;} \\
 p \cdot \prod_{k \ne i} (\lambda(e_i)-\lambda(e_k))^{-(p+1)}   \mbox{ if  } \text{ i=j} .   
 \end{array}
 \right. $$ }

\end{enumerate}
\end{theorem} 
\begin{rem}
One can also prove an analogue of the theorem above when $N=3$ and $p \equiv 1 \text{ (mod }3\text{)}$, for a suitable model $\mathfrak{X}$ over $\Z_p$ of the modular curve of level $\Gamma_0(p) \cap \Gamma(3)$ (\textit{c.f.} \ref{N=3} in appendix).
\end{rem}

\textit{Acknoledgements.}
The first author has received funding from the European Research Council (ERC) under the European Union's Horizon 2020 research and innovation programme (grant agreement No 682152).
The second author has received funding from Universit\'e Paris $7$ for his Phd thesis. He would like to thanks his Phd advisor Lo\"ic Merel for drawing his attention to this problem. Both authors are grateful to E. de Shalit for a useful conversation about the reduction map. Finally, the authors would like to thank the anonymous referee for a careful reading and helpful suggestions.

\section{Coarse moduli spaces of moduli stacks of generalized elliptic curves with level structure}
Let $p \geq 5$ be a prime number and $N \geq 2$ be a squarefree integer prime to $p$. Let $\mathfrak{M}_{\Gamma_0(p) \cap \Gamma(N)}$ be the stack over $\Z[1/N]$ whose $S$-points are the isomorphism classes of generalized elliptic curves $E/S$, endowed with a locally free subgroup $A$ of rank $p$ such that $A+E[N]$ meets each irreducible component of any geometric fiber of $E$ ($E[N]$ is the subgroup of $N$-torsion points of $E$) and a basis of the $N$-torsion (\textit{i.e.} an isomorphism $\alpha_N: E[N] \simeq (\Z/N\Z)^2$). Deligne and Rapoport proved in \cite{D-R} that $\mathfrak{M}_{\Gamma_0(p) \cap \Gamma(N)}$ is a regular algebraic stack, proper, of pure dimension $2$ and flat. 

Let $M_{\Gamma_0(p) \cap \Gamma(N)}' $ be the coarse space of the algebraic stack $\mathfrak{M}_{\Gamma_0(p) \cap \Gamma(N)}$ over $\Z[1/N]$. Since $M_{\Gamma_0(p) \cap \Gamma(N)}' $ is smooth and proper over $\Z[1/Np]$, $M_{\Gamma_0(p) \cap \Gamma(N)}'[1/Np]$ is a scheme of relative dimension one over $\Z[1/Np]$ ($M_{\Gamma_0(p) \cap \Gamma(N)}'[1/Np]$ is a quasi-projective scheme). Let $M_{\Gamma_0(p) \cap \Gamma(N)}$ be the model over $\Z[1/N]$ of $M_{\Gamma_0(p) \cap \Gamma(N)}'[1/Np]$ given by the normalization along the modular invariant $j: M_{\Gamma_0(p) \cap \Gamma(N)}'[1/Np] \rightarrow \mathbf{P}^1_{\Z[1/N]}$ (see \cite[IV.3.3]{D-R}). By using the same arguments as in Theorems \cite[V.1.6]{D-R}, \cite[IV.3.4]{D-R}, and Proposition \cite[IV.3.10]{D-R}, we deduce that $M_{\Gamma_0(p) \cap \Gamma(N)}$ is the coarse moduli space of the algebraic stack $\mathfrak{M}_{\Gamma_0(p) \cap \Gamma(N)}$ (\textit{c.f.} variante \cite[V.1.14]{D-R}).

The results of \cite{D-R} show that the scheme $M_{\Gamma_0(p) \cap \Gamma(N)}$ is smooth over $\Spec \Z[1/N]$ outside the points associated to supersingular elliptic curves in characteristic $p$.

The cusps of $M_{\Gamma_0(p) \cap \Gamma(N)}$ correspond to $N$-gons or $Np$-gons and are given by sections $\Spec \Z[1/N, \zeta_N] \rightarrow \mathfrak{M}_{\Gamma_0(p) \cap \Gamma(N)}$ composed with the coarse moduli map $\mathfrak{M}_{\Gamma_0(p) \cap \Gamma(N)} \rightarrow M_{\Gamma_0(p) \cap \Gamma(N)}$.

Meanwhile, the work \cite{K-M} is a reference about moduli of elliptic curves, and much of the material of \cite{D-R} is best dealt in \cite{K-M}.

\begin{rem}
The genus of the complex modular curve $\Gamma(N)\backslash \mathcal{H}\cup \mathbf{P}^{1}(\Q)  $ is $0$ if and only if $N \in \{1,2,3,4,5\}$.
\end{rem}

In order to apply Mumford's uniformization theorem, we need the following result.

\begin{prop}\label{semi-stable}\

\begin{enumerate}
\item The special fiber of $M_{ \Gamma(2)} \otimes \mathcal{O}_K$ is isomorphic to $\mathbf{P}^1_k$.
\item The scheme $\mathfrak{X}$ is regular.
\item The lambda modular invariant induces an isomorphism $$\lambda:M_{ \Gamma(2)} \otimes {\mathbf{Q}} \simeq \mathbf{P}^{1}_\Q.$$

\end{enumerate}
\end{prop}

\begin{proof}\
 
i) The scheme $M_{\Gamma(2)}$ is smooth, proper and of relative dimension one over $\Spec \Z[1/2]$, so it is flat and the genus of the fibers is constant (see \cite[7.9]{EGA}). Corollaries \cite[IV.5.5]{D-R} and \cite[IV.5.6]{D-R} show that the geometric fibers of $M_{\Gamma(2)}$ over $\Spec \Z[1/2]$ are connected and smooth, hence irreducible. Thus, the fact that the coarse moduli space formation commutes with flat base change implies that $$M_{\Gamma(2)} {\small \otimes \C} \simeq  \Gamma(2) \backslash (\mathcal{H}\cup \mathbf{P}^{1}(\Q)).$$ 

Hence the genus of the special fiber of $M_{\Gamma(2)}$ at $p$ is $0$. On the other hand, the cusps of $M_{\Gamma(2)}$ correspond to $2$-polygons or $2p$-polygons and are given by sections $\Spec \Z[1/2] \rightarrow \mathfrak{M}_{\Gamma(2)}$, then $M_{\Gamma(2)}{ \small \otimes k}$ has a $k$-rational point. Thus, the curve $M_{\Gamma(2)}{ \small \otimes k}$ is isomorphic to the projective line $\mathbf{P}^1_{k}$.

ii) Note that a local noetherian ring is regular if and only if its strict henselianization is regular. Let $x:\Spec(k(x)) \rightarrow \mathfrak{X}$ be a singular geometric point (\textit{i.e.} it corresponds to a supersingular elliptic curve). Thanks to Theorem \cite[VI.6.9]{D-R}, the completion of the strict henselianization of the local ring of $\mathfrak{X}$ at $x$ is isomorphic to ${ \scriptsize W(\bar{k})[[X,Y]]/(XY-p^{n})}$, where $n$ is the cardinality of the automorphism group$\mod \{-1,1\}$ of the pair $(E,\alpha_2)$, $E$ and $\alpha_2$ are respectively a supersingular elliptic curve and a basis of the $2$-torsion of $E$ associated to the geometric points $x$. But any automorphism of $E$ different from $\{-1,1\}$ acts non trivially on $\alpha_2$, so $n=1$.

iii) Let $E_{\lambda}$ be the elliptic curve over $\mathbf{A}^{1}_\Q$ given by the equation $$Y^2=X(X-1)(X-\lambda) \text{ where } \lambda \in \Q[T] \text{ .}$$ 
It is clear that $E_{\lambda}$ induces a natural morphism $\mathbf{A}^{1}_\Q \rightarrow \mathfrak{M}_{\Gamma(2)} \otimes {\Q} $ and after composing with the coarse moduli map, we get a morphism $g: \mathbf{A}^{1}_\Q \rightarrow M_{\Gamma(2)}{ \small \otimes{\Q}} $ which extends to a morphism $\tilde{g}: \mathbf{P}^{1}_\Q \rightarrow M_{\Gamma(2)}{ \small \otimes{\Q}} $. 

Moreover, the universal properties of the coarse moduli space of $M_{\Gamma(2)}$ imply that $\mathfrak{M}_{ \Gamma(2)}(K) \simeq M_{\Gamma(2)}(K)$ for any field $K$. Hence, $\tilde{g}$ is an isomorphism and $\lambda$ is the inverse of $\tilde{g}$.
\end{proof}

\begin{rem}\label{champ N=3}
When $N = 3$, the algebraic stack $\mathfrak{M}_{\Gamma_0(p) \cap \Gamma(N)}$ is rigid. Hence it will be represented by a scheme of relative dimension $1$ over $\Z[1/N]$, which is smooth over $\Z[1/Np]$.

The geometric fibers of the morphism $M_{\Gamma(3)}[1/3] \rightarrow \Spec \Z[1/3,\zeta_3]$ have genus $0$ and are smooth and irreducible. However, they are not geometrically irreducible over $\Spec \Z[1/3]$ (see \cite{D-R}). On the other hand, Deligne and Rapoport proved in \cite[VI.6.8]{D-R} that for any subgroup $K \subset \SL_2(\Z/3\Z)$, the geometric fibers of the morphism $M_K[1/3] \rightarrow \Spec \Z[1/3,\zeta_3]^K$ are of genus $0$, irreducible and smooth. Thus, if we choose $K$ such that $\Z[1/3,\zeta_3]^K=\Z[1/3]$ and $M_K[1/3](\C)=\Gamma(3)\backslash (\mathcal{H}\cup \mathbf{P}^{1}(\Q))$, then Theorem \cite[VI.6.9]{D-R} and the same arguments as in the above propositions show that $M_{K \cap \Gamma_0(p)}\otimes \mathcal{O}_K$ is regular, semi-stable and the irreducible components of its special fiber are $k$-rational.
\end{rem}

\bibliographystyle{siam}

\section{$p$-adic uniformization of $\mathfrak{X}_K^{an}$ and the reduction map}

Mumford's Theorem \cite{M} shows the existence of a free discrete subgroup $\Gamma \subset \PGL_2(K)$ (\textit{i.e.} a Schottky group) and of a $\Gal(\bar{K}/K)$-equivariant morphism of rigid spaces:$$\tau: \mathfrak{H}_{\Gamma} \rightarrow \mathfrak{X}_K^{an}$$ inducing an isomorphism $\mathfrak{X}_K^{an} \simeq \mathfrak{H}_{\Gamma}/\Gamma$, where $\mathfrak{H}_{\Gamma}=\mathbf{P}^1_K - \mathcal{L}$ and $\mathcal{L}$ is the set of limit points of $\Gamma$. Note that $\mathfrak{H}_{\Gamma}$ is an admissible open of the rigid projective line $\mathbf{P}^1_K$. 

Let $\mathcal{T}_{\Gamma}$ be the subtree of the Bruhat--Tits tree for $\PGL_2(K)$ generated by the axes whose ends correspond to the limit points of $\Gamma$. 
Mumford constructed in \cite{M} a continuous map $\rho: \mathfrak{H}_{\Gamma} \rightarrow \mathcal{T}_{\Gamma}$ called the reduction map, satisfying the following properties:

\begin{prop}\label{red1}\

\begin{enumerate}

\item The inverse image of any vertex $v$ of $\mathcal{T}_{\Gamma}$ is isomorphic to the closed unit disk minus the union of some open maximal rational sub-disks.

\item The inverse image of any edge is an open annulus of width $r \in |K^{\times}|_p$ \textit{ i.e.}
$$ D(0, r[ - D(0,1] \text{ .}$$

\item If we endow the graph $\mathcal{T}_{\Gamma}$ with the natural topology in which we identify an edge to the interval $[0,1]$, then $\rho$ is continuous.

\item If $\Delta$ is the dual graph of the special fiber of $\mathfrak{X}$, then $\Delta \simeq \mathcal{T}_{\Gamma}/\Gamma$ and $\rH_1(\Delta,\Z) \simeq \Gamma^{ab}\simeq \Z[S]^0$, where $\Z[S]^0$ is the augmentation subgroup of $\Z[S]$.

\end{enumerate}

\end{prop}

The following proposition follows immediately from the the proof of the main theorem of Mumford in \cite{M}.

\begin{prop}\label{red2}\
Let $v$ be a vertex of $\mathcal{T}_{\Gamma}$ having $k+1$ neighbours denoted by $v_i$ for $1\leq i \leq k+1$. Then $\rho^{-1}(v)$ is isomorphic to the closed unit disk $D(0,1]$ minus $k$ maximal open disks $D(a_i,1[$. Each disk $D(a_i,1[$ contains the inverse image by $\rho$ of a unique neighbour $v_i$ of $v$, \textit{i.e.} $\rho^{-1}(v_i)\subset D(a_i,r_i]\subset D(a_i,1[$ for a unique neighbour $v_i$ of $v$. There also exists a unique neighbour  $v_{i_0}$ of $v$ such that  $\rho^{-1}(v_{i_{0}}) \subset D(0,r_{i_0}]$ where $r_{i_{0}}<1$, and the annuli $D(a_i,1[-D(a_i,r_i]$ for $i\ne i_0$ or $D(0,1[-D(0,r_{i_0}]$ reduce respectively to the edges connecting $v$ to $v_i$ and $v$ to $v_{i_0}$.

\end{prop}

The special fiber of $\mathfrak{X}$ has two components, and each component has $3$ cusps. 
One of these components, which we call the \textit{\'etale} component, classifies elliptic curves or $2p$-sided N\'eron polygons over $\bar{k}$ with an \'etale subgroup of  order $p$ and a basis of the $2$-torsion. The other component, which we call the \textit{multiplicative} component, classifies elliptic curves or $2$-sided N\'eron polygons over $\bar{k}$ with a multiplicative subgroup of order $p$ and a basis of the $2$-torsion.

The involution $w_p$ sends a $2p$-gon to a $2$-gon, since the quotient of a $2p$-gon by its unique cyclic \'etale subgroup of order $p$ (\textit{i.e.} $\Z/p\Z$) in its smooth locus gives $2$-sided polygone ($\Z/p\Z$ acts by rotations).  

Let $c$ and $c'=w_p(c)$ be two cusps of $M_{\Gamma_0(p)\cap \Gamma(2)}(\C)$ defined in Appendix, section \ref{section_cusps_N=2}. They are $\Q$-rational and $M_{\Gamma_0(p)\cap \Gamma(2)}$ is proper over $\Z[1/2]$. The valuative criterion shows that there exists two cusps $\xi_c :\Spec \Z_p \rightarrow \mathfrak{X}$ (resp. $\xi_{c'}: \Spec \Z_p \rightarrow \mathfrak{X}$) corresponding to the cusps $c$ and $c'$ of $M_{\Gamma_0(p)\cap \Gamma(2)}(\C)$ after taking the generic fiber. The cusp $\xi_{c}$ corresponds to a $2p$-gon and $\xi_{c'}=w_p(\xi_c)$ corresponds to a $2$-gon. 

The dual graph $\Delta$ of the special fiber of $\mathfrak{X}$ has two vertices $v_{c'}$ and $v_{c}$ indexed respectively by the cusps $\xi_{c'}$ and $\xi_c$. There are $g+1$ edges $e_i$ ($i \in \{0,...,g\}$) corresponding to supersingular elliptic curves with a $\Gamma(2)$-structure. We orient these edges so that they point out of $v_{c'}$.

The Atkin-Lehner involution $w_p$ exchanges the two vertices $v_{c'}$ and $v_{c}$ and also acts on edges (reversing the orientation). More precisely, if $E_i$ is a supersingular elliptic curve corresponding to $e_i$, then $w_p(e_i)=e_j$ where $e_j$ is the elliptic curve associated to $E_i^{(p)}=w_p(E_i)$ (here $w_p$ is the Frobenius). Thanks to Lemma \ref{normalization} below, one can identify the generators $\{\gamma_i\}_{1\leq i\leq g+1}$ of $\Gamma$ with $(e_i-e_0)_{1\leq i \leq g+1}$. 

Let $\tilde{v}_{c}$ and $\tilde{v}_{c'}$ be two neighbour vertices of $\mathcal{T}_{\Gamma}$ reducing to $v_{c'}$ and $v_{c}$ respectively, such that the edge linking $\tilde{v}_{c}$ to $\tilde{v}_{c'}$ reduces to $e_0$ modulo $\Gamma$. For $0 \leq i \leq g$, let $\tilde{e}_i'$ be an edge pointing out of $\tilde{v}_{c'}$ and reducing to $e_i$ modulo $\Gamma$. Let $\tilde{e}_i$ be oriented edges of $\mathcal{T}_{\Gamma}$ lifting $e_i$ and pointing to $\tilde{v}_c$. Note that $\tilde{e}_0 = \tilde{e}_0'$.

 Let $A=\rho^{-1}(\tilde{v}_{c})$ and $A'=\rho^{-1}(\tilde{v}_{c'})$.  Proposition \ref{red2} implies that $A$ (resp. $A'$) is the complement of $g+1$ open disks in $\mathbf{P}^1_K$, hence $\mathbf{P}^1_K - A= \underset{0\leq i \leq g}{\coprod} B_i$ and $\mathbf{P}^1_K - A'=\underset{0\leq i \leq g}{\coprod} C'_i$ where $0 \leq i \leq g$. We index $B_i$ and $C'_i$ such that $A \subset C'_0$, $A' \subset B_0$, $B_i$ and $C'_i$ are associated to $\tilde{e}_i'$ and the inverse of $\tilde{e}_i$ respectively.

For all $0\leq i \leq g$, $\rho^{-1}(\tilde{e}_i')=c_i$ is an annulus of $C'_i$ and $C_i=C'_i-\rho^{-1}(\tilde{e}_i')$ is a closed disk; we also have $\mathbf{P}^1_K-C_0=B_0$. We have $$\mathbf{P}^1_K - \rho^{-1}(\underset{0 \leq i \leq g}{\cup} \tilde{e_i}' \cup \tilde{v}_{c'})=\underset{0\leq i \leq g}{\coprod} C_i \text{ .}$$

Note that $\tilde{v}_{c} \cup \tilde{v}_{c'} \cup_i \{\tilde{e}_i'\}$ is a fundamental domain of $\mathcal{T}_{\Gamma}$, so $$D=\mathbf{P}^1_K-\underset{1\leq i \leq g}{\coprod} B_i \cup \underset{1\leq i \leq g}{\coprod} C_i$$ is a fundamental domain of $\mathfrak{H}_{\Gamma}$. 

Using the same techniques as in \cite[\S1]{shalit}, we get the following lemma:

\begin{lemma}[Normalization of $\Gamma$]\label{normalization}
We can choose $\Gamma$ such that there is a Schottky basis $\alpha_1,..., \alpha_g$ of $\Gamma$,  and a fundamental domain $D$ satisfying:

\begin{enumerate}

\item $B_i$ is the open residue disk in the closed unit disk of $\mathbf{P}^1_K$ which reduces to $\lambda(e_i)^p$, $\forall 0\leq i \leq g$.

\item For $1\leq i \leq g$, $\alpha_i$ corresponds, under the identification $\Gamma^{ab} = \Z[S]^0$, to $e_i-e_0$.

\item $\alpha_i$ sends bijectively $\mathbf{P}^1_K-B_i$ to $C_i$ and $\alpha_i^{-1}$ sends bijectively $\mathbf{P}^1_K-C_i$ to $B_i$.

\item The annulus $c_i$ is isomorphic, as a rigid analytic space, to $\{z , |p| < |z|< 1\}$.

\end{enumerate}

\end{lemma}

\begin{proof}\
\begin{enumerate}

\item The standard reduction sends $\mathbf{P}^1_K $ to $\mathbf{P}^1_k$ (see \cite[II.2.4.2]{W-L}). If we restrict the standard reduction to $A$, we get back $\rho$, so $\rho$ sends $A$ to $\mathbf{G}_m$ minus the $\rho(B_i\cap \mathfrak{H}_{\Gamma})=\lambda(e_i)^p$'s.
\item $\mathcal{T}_{\Gamma}$ is the universal covering of $\Delta$, so $\pi_1(\Delta) \simeq \Gamma$. Let $\tilde{v}_i$ be a neighbour of $\tilde{v}_{c'}$ whose edge to $\tilde{v}_{c'}$ lifts $e_i$. By monodromy, we can choose $\alpha_i$ which lifts $e_i-e_0$, such that $\alpha_i^{-1}(\tilde{v}_i)=\tilde{v}_{c}$, and since the action of $\Gamma$ on the graph is continuous, we see that $\alpha_i(\mathbf{P}^1_K - B_i)=C_i$ (\textit{i.e.} $\alpha_i^{-1}$ sends neighbours of $\tilde{v}_i$ to neighbours of $\tilde{v}_{c}$). This also shows (iii).

\end{enumerate} 

\end{proof}

\section{Extension of $\Phi$ to $\mathbf{Z}[S] \times \mathbf{Z}[S]$}
For $a, b \in \mathfrak{H}_{\Gamma}$, define the meromorphic function $\theta(a,b; z) = \theta((a)-(b); z)$ ($z \in \mathfrak{H}_{\Gamma}$) by the convergent product 
$$ \theta(a,b;z)= \prod_{\gamma \in \Gamma} \frac{z-\gamma a}{z-\gamma b} \text{ .}$$

See \cite{D} for the basic properties of these theta functions. 

For all $a,b \in \mathfrak{H}_{\Gamma}$, the theta series $\theta(a,b,.)$ converges and defines a rigid meromorphic function on $\mathfrak{H}_{\Gamma}$ (which is modified by a constant if we conjugate $\Gamma$). We extend $\theta$ to degree zero divisors $D$ of $\mathcal{H}_\Gamma$. The series $\theta(D,.)$ is entire if and only if $\tau_{*}(D)=0$, where we recall that $\tau : \mathfrak{H}_{\Gamma} \rightarrow \mathfrak{X}^{an}_K$ is the uniformization. 

The proposition below follows from \cite{D} (see also \cite{shalit}).

\begin{prop}\label{theta}\cite{D}
\begin{enumerate}
\item $\theta(a,b,z)=c(a,b,\alpha) \theta(a,b,\alpha z)$, where $\alpha \in \Gamma$ and $c(a,b,\alpha \beta)=c(a,b,\alpha)c(a,b,\beta)$.
\item The function $u_{\alpha}(z)=\theta(a,\alpha a,z)$ does not depend on $a$, and $u_{\alpha \beta}=u_{\alpha} u_{\beta}$.
\item $c(a,b,\alpha)= u_{\alpha}(a)/u_{\alpha}(b)$.
\item $\theta(a,b,z)/\theta(a,b,z')=\theta(z,z',a)/\theta(z,z',b)$.

\end{enumerate}
\end{prop}

We recall that $\Phi:\mathbf{Z}[S]^0 \times \mathbf{Z}[S]^0 \rightarrow K^{\times}$ is defined by:

$$\Phi (\alpha, \beta)=\theta(a,\alpha a,z)/\theta(a,\alpha a, \beta z)=u_{\alpha}(z)/u_{\alpha}(\beta z).$$

We can identify $\Gamma^{ab}$ with $\Z[S]^0$ since $\cT_{\Gamma}$ is the universal covering of the graph $\Delta$. 
 
 Manin and Drinfeld proved that $v_K \circ \Phi$ is positive definite ($v_K$ is the $p$-adic valuation of $K$).

\begin{lemma}

The pairing $\Phi$ takes values in $\Q_p^{\times}$.

\end{lemma}

\begin{proof}
Since $p\equiv 1 \text{ (mod } 2 \text{)}$, the Atkin--Lehner automorphism $w_p$ is an involution which is the Frobenius on the supersingular points. The proof is now the same as in Lemma \cite[0.4]{shalit}.
\end{proof}

Note that our pairing $\Phi:\mathbf{Z}[S]^0 \times \mathbf{Z}[S]^0 \rightarrow \Q_p^{\times}$ does not depend on the choice of $\lambda:X(2)_\Q = M_{\Gamma(2)} \otimes \Q \xrightarrow{\sim} \mathbf{P}^1_\Q$. We are going to define an extension $\Phi: \mathbf{Z}[S] \times \mathbf{Z}[S] \rightarrow K^{\times}$ (which depends on the choice of $\lambda$). 

\subsection{Extension of $\Phi$ to $\mathbf{Z}[S]^0 \times \mathbf{Z}[S]$}
For all $0 \leq i \leq g$, we choose $\xi_c^{(i)}$ (resp. $\xi_{c'}^{(i)}$) in $\mathfrak{H}_{\Gamma}$ which reduces modulo $\Gamma$ to the cusp $\xi_c \otimes \Q_p$ (resp. $\xi_{c'} \otimes \Q_p$), and such that $\xi_c^{(i)}$ and $\xi_{c'}^{(i)}$ are separated by an annulus reducing to $e_i$.  

Let $\tilde{v}_c^{(i)}$ and $\tilde{v}_{c'}^{(i)}$ be two neighbour vertices of $\mathcal{T}_{\Gamma}$ above $v_c$ and $v_{c'}$ respectively, separated by an edge reducing to $e_i$. We fix $\tilde{v}_c^{(0)} = \tilde{v}_c$ and $\tilde{v}_{c'}^{(0)} = \tilde{v}_{c'}$. We then choose $\xi_c^{(i)}$ (resp. $\xi_{c'}^{(i)}$) in $\rho^{-1}(\tilde{v}_c^{(i)})$ (resp.  $\rho^{-1}(\tilde{v}_{c'}^{(i)})$). If we choose, for all $0 \leq i \leq g$, $\xi_c^{(i)} = z_0 \in A$, then the $\xi_{c'}^{(i)}$ are uniquely determined, and satisfy
$$\xi_{c'}^{(i)} = \alpha_i^{-1}(\xi_{c'}^{(0)}) \in B_i \text{ .}$$
Indeed, we have $\xi_{c'}^{(0)} \in \rho^{-1}(\tilde{v}_{c'})=A'$ and $\alpha_i^{-1}(A') \subset \alpha_i^{-1}(\mathbf{P}^1 - C_i') \subset B_i$.

Note that for all $i$, the pair $(\xi_c^{(i)}, \xi_{c'}^{(i)})$ is uniquely determined modulo $\Gamma$. We can assume that $z_0 \neq \infty$.

Let $a \in \mathfrak{H}_{\Gamma}$. We then define, for all $\alpha \in \Gamma$, 
\begin{equation}
\Phi(\alpha, e_i) = \frac{\theta(a, \alpha(a), \xi_{c'}^{(i)})}{\theta(a, \alpha(a), \xi_{c}^{(i)})}=\frac{u_{\alpha}(\xi_{c'}^{(i)})}{u_{\alpha}(\xi_{c}^{(i)})}=\frac{u_{\alpha}(\xi_{c'}^{(i)})}{u_{\alpha}(z_0)}
\label{def_Q_1}
\end{equation}
This definition does not depend on the choice of $a$ and $(\xi_c^{(i)}, \xi_{c'}^{(i)})$, by Proposition \ref{theta}. Since $K$ is complete and $\xi_c^{(i)}$ and $\xi_{c'}^{(i)}$ are defined over $K$, $\Phi$ takes values in $K^{\times}$.
\begin{lemma}
The pairing $\Phi$ defined above extends the previous pairing $\Phi$ on $\Z[S]^0 \times \Z[S]^0$.
\end{lemma}
\begin{proof}
Using Proposition \ref{theta}, we have: $$\frac{\Phi(\alpha, e_i)}{\Phi(\alpha, e_0)} = \frac{u_{\alpha}(\xi_{c'}^{(i)})}{u_{\alpha}(\xi_{c'}^{(0)})} = \frac{u_{\alpha}(\xi_{c'}^{(i)})}{u_{\alpha}\left(\alpha_i(\xi_{c'}^{(i)})\right)}  = \Phi(\alpha, \alpha_i) = \Phi(\alpha, e_i-e_0)\text{ .}$$
\end{proof}

\section{Extension of $\Phi$ to $\Z[S] \times \Z[S]$}

Let $\lambda$ be the Hauptmodul for $M_{\Gamma(2)}$; recall that $\lambda$ induces an isomorphism $$M_{\Gamma(2)} \otimes \mathbf{Q} \simeq \mathbf{P}^{1}_{\mathbf{Q}} \text{ .}$$ 
We define $\lambda' :\mathfrak{X}_K^{an}\rightarrow \mathbf{P}^1_K$ by $\lambda' = \lambda \circ w_p$

\subsection{Atkin--Lehner involution on $\mathfrak{H}_{\Gamma}$}
The Atkin--Lehner involution acts on $\Gamma \backslash \mathcal{T}_{\Gamma}$ and lifts to an orientation reversing involution $w_p$ of $\mathcal{T}_{\Gamma}$ (by the universal covering property). By \cite{G} ch. VII Sect. $1$, there is a unique class in $N(\Gamma)/\Gamma$ (where $N(\Gamma)$ is the normalizer of $\Gamma$ in $\PGL_2(K)$) inducing $w_p$ on $\mathcal{T}_{\Gamma}$. We denote by $w_p$ the induced map of $\mathfrak{H}_{\Gamma}$ (it is only unique modulo $\Gamma$). 

\subsection{Definition of $\Phi$}
Fix $0 \leq i , j \leq g$. 

Let $z \in \mathfrak{H}_{\Gamma}$ near $\xi_{c}^{(i)}$ and $z'$ near $\xi_{c'}^{(i)}$ such that $\tau(z) = w_p(\tau(z'))$. Recall that by hypothesis, $\xi_{c}^{(i)}=z_0$ is independent of $i$.

For $0 \leq i,j \leq g$, we let:
\begin{equation}
\Phi(e_i,e_j) = \text{lim } \lambda'(\tau(z))^2 \cdot \frac{\theta(z',z,\xi_{c'}^{(j)})}{\theta(z',z,\xi_{c}^{(j)})}
\label{def_Q_2}
\end{equation}
where $z$ and $z'$ approach $\xi_{c}^{(i)}$ et $\xi_{c'}^{(i)}$ respectively.
Since at $z = z_0$, $\lambda' \circ \tau$ has a simple pole and the numerator and denominator have a simple zero and simple pole respectively, $\Phi(e_i,e_j)$ is finite, and is in $K^{\times}$ since $K$ is complete (we choose $z$ and $z'$ in $K$ to compute the limit). We bilinearly extend $\Phi$ to $\Z[S] \times \Z[S]$.

\begin{lemma}
The pairing $\Phi$ defined above extends the previous pairing on $\Z[S] \times \Z[S]^0$.\end{lemma}
\begin{proof}
Since $\xi_c^{(i)} = \xi_c^{(0)} = z_0$, we have: 
$$\frac{\Phi(e_i, e_j)}{\Phi(e_0, e_j)} =  \text{lim } \frac{\theta(z',z,\xi_{c'}^{(j)})}{\theta(\zeta',z,\xi_{c'}^{(j)})} \cdot \frac{\theta(\zeta', z, \xi_{c}^{(j)})}{\theta(z',z,\xi_{c}^{(j)})} $$
where $\zeta' = \alpha_i(z')$ approaches $\xi_{c'}^{(0)}$. Since $$\frac{\theta(a,b,z)}{\theta(\gamma(a),b, z)} = \theta(a, \gamma(a), z)$$ (which is obvious by the infinite product definition of $\theta$), we get:
$$\frac{\Phi(e_i, e_j)}{\Phi(e_0, e_j)} =  \text{lim } \frac{\theta(z',\alpha_i(z'),\xi_{c'}^{(j)})}{\theta(z',\alpha_i(z'),\xi_{c}^{(j)})} = \Phi(\alpha_i, e_j) \text{ .}$$
\end{proof}

\section{Proof of the main theorem}
Let $\hat{e} = \sum_{e_i \in S} [e_i] \in \Z[S]$ be \textit{the Eisenstein element}. Let $\bar{\Phi}$ be the reduction of $\Phi$ modulo principal units. Let $d = \text{gcd}(p-1, 12)$.
\begin{theorem}\label{thm_principal}
If $i \neq j$, we have:
\begin{equation}
\overline{\Phi}(e_i,e_j)^{\frac{12}{d}}=(\lambda(e_i)-\lambda(e_j))^{(p+1)\cdot \frac{12}{d}} \text{ .}
\label{thm_ij}
\end{equation}
Else, we have:
\begin{equation}
\bar{\Phi}(e_i,\hat{e})^{\frac{12}{d}} =  p^{\frac{12}{d}} \text{ .}
\label{thm_i}
\end{equation}
\end{theorem}
Theorem \ref{main-thm} follows from Theorem \ref{thm_principal} since $\mathbf{Q}_p^{\times}/U_1(\mathbf{Q}_p)$ has no $\frac{12}{d}$-torsion.
In the rest of this article, we prove Theorem \ref{thm_principal}.
\subsection{Case $i \neq j$}

Let's first show (\ref{thm_ij}), which is easier. Fix $0 \leq i \leq g$. For simplicity, but without loss of generality, we assume that $j=0$.

Proposition \ref{theta} (iv) shows that
$$\Phi(e_i, e_0) = \text{lim } \lambda'(\tau(z))^2 \cdot \frac{\theta(\xi_{c'}^{(0)}, \xi_c^{(0)}, z')}{\theta(\xi_{c'}^{(0)}, \xi_c^{(0)}, z)} \text{ .}$$ 
Recall that, by definition, $\xi_c^{(0)}=z_0 \in A$.
For each $\gamma \in \Gamma$, the corresponding term in the infinite product defining the above expression is:
$$\frac{(z' - \gamma(\xi_{c'}^{(0)}) )\cdot ( z - \gamma(z_0) )}{(z'-\gamma(z_0) )\cdot (z-\gamma(\xi_{c'}^{(0)}) )}$$
\begin{lemma}\label{H' identity}
The map $\lambda' \circ  \tau:\mathfrak{H}_{\Gamma} \rightarrow \mathbf{P}^1_{K}$ stabilizes $A$ and $A \cap U$ where $U$ is the closed unit disk. Furthermore, the restriction of $\lambda' \circ \tau$ to $A \cap U$ is the identity modulo $p$.
\end{lemma}
\begin{proof}
We have a map $A \rightarrow \mathbf{P}^1_k - S$ given by composing the uniformization $\tau$ with the reduction on the special fiber $\mathfrak{X}_k$. This map coincides with the naive reduction of $A$ to $\mathbf{P}^1_k$ induced by the standard reduction $\mathbf{P}^1_K \rightarrow \mathbf{P}^1_k$. Thus, by the lemma of normalization of $\Gamma$, $\lambda' \circ \tau$ stabilizes $A$ and is the identity modulo $p$. 
\end{proof}
Since $\lambda'(\xi_c)=\infty$, $\xi_{c}^{(0)} = z_0 \in A$ reduces to $\infty$ modulo $p$, \textit{i.e.} is not $p$-integral as an element of $K$. Since $z$ is near $z_0$, $z$ reduces to $\infty$ modulo $p$. 

\subsubsection{Case $\gamma=1$}
If $\gamma=1$, $$\frac{(z' - \gamma(\xi_{c'}^{(0)}) )\cdot ( z - \gamma(z_0) )}{(z'-\gamma(z_0) )\cdot (z-\gamma(\xi_{c'}^{(0)}) )}$$ is equivalent, modulo principal units, to 
$$\frac{(\lambda(e_i)^p-\lambda(e_0)^p)\cdot (z-z_0)}{-z_0^2} = \frac{(\lambda(e_i)-\lambda(e_0))^p\cdot (z-z_0)}{-z_0^2}\text{ .}$$
Note that this equality mixes terms in $k^{\times}$ and $K^{\times}$. In fact, this equality (and all the similar equalities below) is viewed in $K^{\times}/U_1(K)$.

\subsubsection{Case $\gamma \neq 1$}
Let $\gamma \neq 1 \in \Gamma$, written in reduced form as $\alpha_{i_1}^{\epsilon_1}\alpha_{i_2}^{\epsilon_2} \cdot \cdot \cdot$ with $i_k \in \{1,...,g\}$ and $\epsilon_{i_k} \in \Z \backslash \{0\}$. Then for all $u$ in the fundamental domain $$D=\mathbf{P}^1_K-\underset{1\leq i \leq g}{\coprod} B_i \cup \underset{1\leq i \leq g}{\coprod} C_i \text{ ,}$$ we have $\gamma(u) \in B_{i_1}$ if $\epsilon_1 < 0$ and $\gamma(u) \in C_{i_1}$ if $\epsilon_1 > 0$. Note that $\gamma(z_0)$ is not in $A$, so it is not in the residue disk containing $\infty$ (\textit{i.e.} the residue disk containing $z_0$). Hence $\gamma(z_0)$ does not reduce to $\infty$ modulo $p$.

In particular, $\gamma(D) \subset B_{i_1}$ or $C_{i_1}$. Thus $$\frac{z - \gamma(z_0) }{z-\gamma(\xi_{c'}^{(0)}) }$$ is a principal unit.

Assume first that $\alpha_j^{k}$ with $i \ne j $ or $\alpha_i ^n$ with $n>0$ occurs as the first term in the reduced form of $\gamma$. In that case, $\gamma(\xi_{c'}^{(0)})$ and $\gamma(z_0)$ are both in $B_j$, $C_j$ or $C_i$ and $z' \in B_i$, so $$\frac{z' - \gamma(\xi_{c'}^{(0)}) }{z'-\gamma(z_0) }$$ is a principal unit. So in this case, the factor associated to $\gamma$ in the definition of $\phi(e_i,e_0)$ is a principal unit.

Thus, in what follows, assume that $\gamma=\alpha_i^{-1}\cdot \gamma'$, where $\gamma'$ is such that its reduced expression in terms of the $\alpha_k$'s does not begin with $\alpha_i$. By invariance of the cross-ratio by the action $\text{PGL}_2(K)$ (here by the action of $\alpha_i$), we have:
$$\frac{(z' - \gamma(\xi_{c'}^{(0)}) )\cdot ( z - \gamma(z_0) )}{(z'-\gamma(z_0) )\cdot (z-\gamma(\xi_{c'}^{(0)}) )} = \frac{(\alpha_i(z') - \gamma'(\xi_{c'}^{(0)}) )\cdot (\alpha_i(z) -\gamma'(z_0) )}{(\alpha_i(z')-\gamma'(z_0) )\cdot (\alpha_i(z)-\gamma'(\xi_{c'}^{(0)}) )}\text{ .}$$

\textit{First subcase.} We first assume that the decomposition of $\gamma'$ does not begin by $\alpha_i^{-1}$.

We distinguish the cases $\lambda(e_0) \in \mathbf{F}_p$ and $\lambda(e_0) \in k \backslash \mathbf{F}_p$. 

\subsubsection*{Case $\lambda(e_0) \in \mathbf{F}_p$}
Assume first that $\lambda(e_0) \in \mathbf{F}_p$.
We choose an element of $N(\Gamma)$ (normalizer of $\Gamma$ in $\text{PGL}_2(K)$) inducing $w_p$ on $\mathfrak{X}_K^{an}$ and such that the induced automorphism on $\mathcal{T}_{\Gamma}$ (still denoted by $w_p$) fixes (by reversing the vertices) the edge $\tilde{e}_0$ between $\tilde{v}_c$ and $\tilde{v}_{c'}$. Then, by definition, $w_p(A)=A'$ and $w_p(C_i')=B_{i'}$ where $i' \in \{0, ..., g\}$ is such that $\lambda(e_{i'}) = \lambda(e_i)^p$. Since $w_p^2$ fixes the edge $\tilde{e_0}$, $w_p^2=1$ (the stabilizer of $\tilde{e_0}$ is compact and $\Gamma$ is discrete, so $w_p^2$ is torsion but $\Gamma$ is torsion-free). We have $w_p(\tilde{v}_{c'}) = \tilde{v}_c$, hence $w_p(\tilde{v}_c) = \tilde{v}_{c'}$. Therefore $w_p(z_0) = \xi_{c'}^{(0)}$ and $w_p(\xi_{c'}^{(0)}) = z_0$. We also have $w_p(\alpha_i(z')) = z$ (both terms are near $z_0$ and $\lambda'(\tau(w_p(\alpha_i(z'))) = \lambda(\tau(\alpha_i(z'))) = \lambda(\tau(z')) = \lambda'(\tau(z))$). By invariance of the cross-ratio by the action of $\text{PGL}_2(K)$ (hence by $w_p$) and if $\gamma' = 1$, we have:
$$\frac{(\alpha_i(z') - \xi_{c'}^{(0)} )\cdot (\alpha_i(z) - z_0 )}{(\alpha_i(z')-z_0 )\cdot (\alpha_i(z)-\xi_{c'}^{(0)}) } =\frac{(z - z_0 )\cdot (w_p(\alpha_i(z)) -\xi_{c'}^{(0)} )}{(z-\xi_{c'}^{(0)})\cdot (w_p(\alpha_i(z))-z_0 )} \text{ .}$$
This last term equals, modulo principal units, to: 
$$ \frac{(z-z_0)}{-z_0^2}\cdot (\lambda(e_{i})-\lambda(e_0)) \text{ .}$$

Otherwise, $\gamma' \neq 1$ and we get similarly:
$$\frac{(\alpha_i(z') - \gamma '( \xi_{c'}^{(0)} ))\cdot (\alpha_i(z) - \gamma '( z_0))}{(\alpha_i(z')-\gamma '( z_0) )\cdot (\alpha_i(z)-\gamma ' (\xi_{c'}^{(0)})) } = \frac{(z- w_p(\gamma'(\xi_{c'}^{(0)})) )\cdot (w_p(\alpha_i( z)) -w_p (\gamma' (z_0) ))}{(z- w_p(\gamma' (z_0)) )\cdot (w_p(\alpha_i (z))- w_p(\gamma'(\xi_{c'}^{(0)}))) } \text{ .}$$
If the reduced form of $\gamma'$ is $\alpha_{j}^k \alpha_{i_2}^{\epsilon_2} \cdot \cdot \cdot$ with $i \ne j$ and $k \geq 1$, then the right-hand side of the equality is seen to be a principal unit, otherwise the reduced form of $\gamma'$ is $\alpha_{j}^{-k} \alpha_{i_2}^{\epsilon_2} \cdot \cdot \cdot$ with $k \geq 1$, so the left-hand side is seen to be a principal unit.

To summarize, in the case $\lambda(e_0) \in \mathbf{F}_p$, if $\gamma' = 1$, the factor associated to $\gamma$ is 
$$ \frac{(z-z_0)}{-z_0^2}\cdot (\lambda(e_{i})-\lambda(e_0)) \text{ .}$$
Else, if $\gamma' \neq 1$, the factor associated to $\gamma$ is a principal unit.

\subsubsection*{Case $\lambda(e_0) \in k \backslash \mathbf{F}_p$}
Assume now that $\lambda(e_0) \in k \backslash \mathbf{F}_p$, and without loss of generality that $\lambda(e_0)^p=\lambda(e_g)$. In that case, we choose $w_p$ such that $w_p(\tilde{e_0}) = \tilde{e_g}$. Since $w_p^2 \in \Gamma$, we have $w_p^2 = \alpha_g$ (this follows from $w_p^2(\tilde{v}_c)=\tilde{v_g} = \alpha_g(\tilde{v}_c)$ and the fact that the stabilizer of a vertex in $\mathcal{T}_{\Gamma}$ is trivial). With the same notation as in the previous case, we have $w_p(B_{i'})=C_i'$, $w_p(z_0) = \xi_{c'}^{(0)}$, $w_p^{-1}(z_0) \in B_g$ et $w_p^{-1}(\alpha_i(z')) = z$. We again use the invariance of the cross-ratio by $w_p^{-1}$ when $\gamma' = 1$:
$$\frac{(\alpha_i(z') - \xi_{c'}^{(0)} )\cdot (\alpha_i(z) -z_0 )}{(\alpha_i(z')-z_0 )\cdot (\alpha_i(z)-\xi_{c'}^{(0)} )} =\frac{(z - z_0 )\cdot (w_p^{-1}(\alpha_i(z)) -\alpha_g^{-1}(\xi_{c'}^{(0)}) )}{(z-\alpha_g^{-1}(\xi_{c'}^{(0)}))\cdot (w_p^{-1}(\alpha_i(z))-z_0 )} \text{ .}$$
This last term equals, modulo principal units, to:
$$ \frac{(z-z_0)}{-z_0^2}\cdot (\lambda(e_{i})-\lambda(e_0))$$

Otherwise, $\gamma' \neq 1$ and by using the invariance of the cross-ratio by the action of $\text{PGL}_2(K)$ (hence by $w_p^{-1}$), we get:
$$\frac{(\alpha_i(z') - \gamma '( \xi_{c'}^{(0)}) )\cdot (\alpha_i(z) - \gamma ' (z_0) )}{(\alpha_i(z')-\gamma '( z_0 ))\cdot (\alpha_i(z)-\gamma ' (\xi_{c'}^{(0)})) } = \frac{(z- w_p^{-1}(\gamma'(\xi_{c'}^{(0)})) )\cdot (w_p^{-1}(\alpha_i (z)) -w_p^{-1} (\gamma' (z_0) ))}{(z- w_p^{-1}(\gamma' (z_0)) )\cdot (w_p^{-1}(\alpha_i (z))- w_p^{-1}(\gamma'(\xi_{c'}^{(0)}))) } \text{ .}$$
Note that $w_p^{-1} \alpha_i z \in B_{i'}$. Thus, if the reduced form of $\gamma'$ is $\alpha_{j}^k \alpha_{i_2}^{\epsilon_2} \cdot \cdot \cdot$ with $i \ne j$ and $k \geq 1$, then the right-hand side of the equality is seen to be a principal unit, otherwise the reduced form of $\gamma'$ is $\alpha_{j}^{-k} \alpha_{i_2}^{\epsilon_2} \cdot \cdot \cdot$ with  $k \leq -1$, so the left-hand side is seen to be a principal unit.

To summarize, in the case $\lambda(e_0) \in k \backslash \mathbf{F}_p$, if $\gamma' = 1$, the factor associated to $\gamma$ is 
$$ \frac{(z-z_0)}{-z_0^2}\cdot (\lambda(e_{i})-\lambda(e_0)) \text{ .}$$
Else, if $\gamma' \neq 1$, the factor associated to $\gamma$ is a principal unit.

\textit{Second subcase. } If the reduced form of $\gamma$ is $\alpha_i^{-n} \gamma'$ with $n\geq 2$, then by using again the invariance of the cross-ratio by the action of $\text{PGL}_2(K)$ (hence by $\alpha_i^n$), we get :
$$\frac{(z' - \gamma(\xi_{c'}^{(0)}) )\cdot ( z - \gamma(z_0) )}{(z'-\gamma(z_0) )\cdot (z-\gamma(\xi_{c'}^{(0)}) )} = \frac{(\alpha^n_i(z') - \gamma'(\xi_{c'}^{(0)}) )\cdot (\alpha_i^n(z) -\gamma'(z_0) )}{(\alpha_i^n(z')-\gamma'(z_0) )\cdot (\alpha_i^n(z)-\gamma'(\xi_{c'}^{(0)}) )}\text{ .}$$

In the case where the first term of $\gamma'$ is $\alpha_j^k$ with $k\leq -1$, a direct computation shows that the right-hand side is a principal unit (since $\alpha_i^n z' \in B_0$). Otherwise, $k\geq 1$ or $\gamma'=1$ and by applying $w_p$ (or $w_p^{-1}$ is the case where $\lambda(e_0) \in k \backslash \mathbf{F}_p$) to the left-hand side above we find a principal unit (in the case where $\lambda(e_0) \in k \backslash \mathbf{F}_p$, $w_p^{-1}(C_g)\subset B_0$). 

Thus in this second subcase, the factor associated to $\gamma$ is a principal unit.

To summarize the various cases, there are exactly two elements $\gamma \in \Gamma$ for which the factor associated to $\gamma$ is not a principal unit, namely $\gamma=1$ and $\gamma = \alpha_i^{-1}$. In the first case, the factor is equivalent modulo principal units to 
$$ \frac{(z-z_0)}{-z_0^2}\cdot (\lambda(e_{i})-\lambda(e_0))^p$$
and in the second case to 
$$ \frac{(z-z_0)}{-z_0^2}\cdot (\lambda(e_{i})-\lambda(e_0)) \text{ .}$$
Thus, in to conclude the proof of (\ref{thm_ij}), it suffices to show: 
\begin{prop}
We have $$ \emph{lim } \frac{\left( \lambda' \circ \tau(z) \right) \cdot (z_0-z)}{z_0^2} = 1\text{ .}$$
\end{prop}
\begin{proof}
This follows from Lemma \ref{H' identity} and \cite{shalit} Lemma $2.1$.
\end{proof}

\subsection{Case $i=j$}

To conclude this article, we prove the formula (\ref{thm_i}) of Theorem~\ref{thm_principal}. Let $d = \text{gcd(}p-1, 12\text{)}$. Let $X = M_{\Gamma_0(p) \cap \Gamma(2)}$ over $\mathbf{Z}[\frac{1}{2}]$.

Let $\mu$ in the function field of $X_\Q=M_{\Gamma_0(p)\cap \Gamma(2)} \otimes \Q$ as in Appendix section \ref{section_mu}. We see $\mu$ as a meromorphic function on $\mathfrak{X}_K^{an}$ (or on $\mathfrak{H}_{\Gamma}$ via the uniformization $\tau : \mathfrak{H}_{\Gamma} \rightarrow \mathfrak{X}_K$, as the context will make clear). It is shown in the appendix that the divisor of $\mu$ is $6\cdot \frac{p-1}{d} \cdot \left( (\xi_{c'}) - (\xi_{c})\right)$.

For $z \in \mathfrak{H}_{\Gamma}$, let
$$\nu(z) = \left( \prod_{0 \leq \ell \leq g} \theta(\xi_{c'}^{(\ell)},z_0,z) \right)^{\frac{12}{d}} \text{ .}$$
\begin{prop}\label{mu_nu}
The two functions $\nu$ et $\mu$ are proportionals.
\end{prop}
Let us show how we conclude the proof of Theorem~\ref{thm_principal} before we prove the above proposition. Let $z$ and $z'$ be as in the definition $\Phi$. 
By definition, we have: 
$$\bar{\Phi}(e_i, \hat{e})^{\frac{12}{d}}= \text{lim } \left( \lambda' \circ \tau(z)\right)^{\frac{12}{d}\cdot 2\cdot\frac{p-1}{2}} \cdot \frac{\nu(z')}{\nu(z)} \text{ .}$$
Proposition \ref{mu_nu} shows that: 
\begin{equation}
\bar{\Phi}(e_i, \hat{e})^{\frac{12}{d}} = \text{lim } \left( \lambda' \circ \tau(z)\right)^{\frac{12}{d}\cdot 2 \cdot \frac{p-1}{2}} \cdot \frac{\mu(z')}{\mu(z)} \text{ .}
\label{limite_algebrique}
\end{equation}
The function $\mu$ is algebraic in $\tau(z)$, and since $\tau(z') = w_p(\tau(z))$, $\frac{\mu(z')}{\mu(z)}$ is algebraic in $\tau(z)$. This is also the case for $\lambda'$. Therefore, in (\ref{limite_algebrique}), it suffices to do a Fourier expansion computation at the cusp $\xi_c$ of $X$ (recall that $\lambda'(\xi_c) = \lambda(\xi_{c'}) = \infty$, by section \ref{section_cusps_N=2} of Appendix)

We know that the Fourier expansion of $\lambda'$ begins with $\frac{-1}{16q_{c}}$, where $q_{c}$ is a local parameter at $\xi_{c}$ (\textit{c.f.} Appendix, section \ref{section_cusps_N=2}). Since $\mu(z') = p^{-\frac{12}{d}} \cdot \mu(z)^{-1}$ by Lemma \ref{functional_equation_mu} and the Fourier expansion of $\mu$ at $\xi_c$ begins with $p^{\frac{-12}{d}} \cdot q_{c}^{-\frac{6(p-1)}{d}}$ by Lemma \ref{Fourier_mu}, the limit we seek is $\left(\left(-\frac{1}{16} \right)^{p-1} \cdot p\right)^{\frac{12}{d}}=p^{\frac{12}{d}}$ modulo principal units. This concludes the proof of Theorem ~\ref{thm_principal}.

Let us now prove Proposition \ref{mu_nu}.

We have seen that both functions have divisor $$\frac{12}{d}\cdot (g+1)\cdot(\Gamma \cdot  \xi_{c'} - \Gamma \cdot \xi_c) =\frac{12}{d}\cdot \frac{p-1}{2}\cdot(\Gamma \cdot  \xi_{c'} - \Gamma\cdot \xi_c) \text{ .}$$ Hence $\frac{\nu}{\mu}$ is entire with no zeroes in $\mathfrak{H}_{\Gamma}$. To conclude, we are going to use the criterion of Corollary $2.4$ of \cite{shalit}. If $\epsilon$ is in an oriented annulus of $\mathfrak{H}_{\Gamma}$ and $f$ is an analytic function on $\mathfrak{H}_{\Gamma}$, the degree $\text{deg}_{\epsilon}(f)$ of the restriction of $f$ to $\epsilon$ was defined by de Shalit in section $2.2$ of \cite{shalit}. It suffices to show that for any edge $e$ of $\mathcal{T}_{\Gamma}$ (with the orientation of $\mathcal{T}_{\Gamma}$ we have fixed before), $\mu$ and $\nu$ have the same degree on the annulus $\rho^{-1}(e)$. Let $\epsilon$ be such an annulus.

It is clear that $\text{deg}_{\epsilon}(\nu) = \frac{12}{d}$. Let $F = \tau(A)$ and $F' = \tau(A')$ be affinoids of $X$. We first show that $\mid \mu(z) \mid =  \mid p \mid^{-\frac{12}{d}}$ for all $z \in F$ and that $\mid \mu(z') \mid =1$ for all $z' \in F'$. Recall that $\mu \circ w_p = p^{\frac{-12}{d}} \cdot \mu^{-1}$ by Lemma \ref{functional_equation_mu}. Therefore it suffices to prove the first norm equality. The affinoid $F$ parametrizes the triples $(E,C,\alpha)$ where $E$ is an elliptic curve with ordinary good reduction at $p$, and $C$ extends to an \'etale group-scheme over some ring of integer. Likewise $F'$ parametrizes triples where $C$ is of multiplicative type. 

If $z \in F$ corresponds to $(E,C,\alpha=(P_1,P_2))$, and $\omega$ is differential on $E$, we let $\omega'$ be the unique differential on $E/C$ whose pullback to $E$ is $\omega$. We then have 
$$\frac{\Delta(pz)}{\Delta(z)} = \frac{\Delta(E/C, p^{-1}\cdot \omega')}{\Delta(E, \omega)} \text{ .}$$
Note that the right hand side is independent of the choice of $\omega$, since if we change $\omega$ by $c\cdot \omega$ then both the numerator and the denominator are multiplied by $c^{12}$.

We now choose $\omega$ to be a N\'eron differential on $E$. Then $\Delta(E, \omega)$ is a unit. Since $C$ extends to an \'etale group scheme over some integer ring, the isogeny $E \rightarrow E/C$ is \'etale and $\omega'$ is also a N\'eron differential on $E/C$. Therefore both $\Delta(E, \omega)$ and $\Delta(E/C, \omega')$ are units, and $\mid u(z) \mid = \mid p \mid^{\frac{-12}{d}}$ where $u$ is defined in Appendix \ref{section_mu}. The complex correspondance $z \mapsto \frac{z+1}{2}$ has a moduli interpretation (see the proof of Lemma \ref{functional_equation_mu}) which shows that the numerator of $\mu$ has norm $\mid p \mid^{\frac{-24}{d}}$ on $F$ and its denominator has norm $ \mid p \mid^{\frac{-12}{d}}$. Thus $\mid u(z) \mid = \mid p \mid^{\frac{-12}{d}}$, as wanted.

Let us finish the proof of Proposition \ref{mu_nu}. Let $f$ be a rigid analytic function on the annulus $\epsilon = \{z, \mid p \mid < z < 1\}$. Then we see, by definition of the degree, that: 
$$\text{deg}_{\epsilon}(f) = \text{lim }_{\mid z \mid \rightarrow \mid p \mid} \text{ord}(f(z)) - \text{lim }_{\mid z \mid \rightarrow 1} \text{ord}(f(z))$$
(where $\text{ord}$ is the $p$-adic valuation normalized by $\text{ord}(p)=1$).
Therefore $\text{deg}_{\epsilon}(\mu)=0-(-\frac{12}{d})=\frac{12}{d}$, which concludes the proof of Proposition \ref{mu_nu}.

\section{Appendix: some basic facts about $M_{\Gamma(N) \cap \Gamma_0(p)}(\mathbf{C})$}

\subsection{Case $N=2$}
\subsubsection{The cusps}\label{section_cusps_N=2}
Let $X = M_{\Gamma_0(p) \cap \Gamma(2)}$ over $\mathbf{Z}[\frac{1}{2}]$. There is an isomorphism 
$$\Gamma_0(p) \cap \Gamma(2) \backslash (\mathcal{H}\cup \mathbf{P}^1(\mathbf{Q}))\simeq X(\mathbf{C})$$
induced by the map sending $z \in \mathcal{H}$ to the triple $(\mathbf{C}/(\mathbf{Z}+z\cdot \mathbf{Z}), <\frac{1}{p}>, (\frac{z}{2}, \frac{1}{2}))$. Let $\epsilon$ be the involution of $X$ defined by
$$\epsilon(E,C, (P,Q)) = (E,C, (-Q, P)) \text{ .}$$
This induces, with the above identification, an automorphism of $\Gamma_0(p) \cap \Gamma(2) \backslash (\mathcal{H}\cup \mathbf{P}^1(\mathbf{Q}))$. The Atkin--Lehner involution on $X(\mathbf{C})$ is induced by the automorphism $z \mapsto \epsilon(\frac{-1}{pz})$ of $\Gamma_0(p) \cap \Gamma(2) \backslash (\mathcal{H}\cup \mathbf{P}^1(\mathbf{Q}))$. Indeed, $\frac{-1}{pz}$ corresponds the triple 
\begin{align*}
(\mathbf{C}/(\mathbf{Z}+ \frac{-1}{pz}\cdot \mathbf{Z}), <\frac{1}{p}>, (\frac{-1}{2pz}, \frac{1}{2})) &= (\mathbf{C}/(z\cdot \mathbf{Z}+ \frac{1}{p}\cdot \mathbf{Z}), <\frac{z}{p}>, (\frac{-1}{2p}, \frac{z}{2}))  \\&=\epsilon(w_p(\mathbf{C}/(\mathbf{Z}+z\cdot \mathbf{Z}), <\frac{1}{p}>, (\frac{z}{2}, \frac{1}{2}))) \text{ .}
\end{align*}

The curve $X(\C)$ has $6 = 2\cdot \text{Card}(\mathbf{P}^1(\mathbf{F}_2))$ cusps, namely $[1]_{\Gamma_0(p) \cap \Gamma(2)}$, $[\frac{1}{2}]_{\Gamma_0(p) \cap \Gamma(2)}$, $[0]_{\Gamma_0(p) \cap \Gamma(2)}$, $[\frac{1}{p}]_{\Gamma_0(p) \cap \Gamma(2)}$, $[\frac{2}{p}]_{\Gamma_0(p) \cap \Gamma(2)}$ and $[\infty]_{\Gamma_0(p) \cap \Gamma(2)}$ (the first three are above the cusp $0$ of $M_{\Gamma_0(p)}(\C)$, and the last three are above $\infty$; $w_p$ exchanges the cusps of the first groups and the cusps of the second group).
The genus of $X$ is $g=\frac{p-3}{2}$, so $X$ has $g+1 = \frac{p-1}{2}$ supersingular points in its special fiber at $p$.
The Hauptmodul $\lambda:X(2) \rightarrow \mathbf{P}^1$ (defined over $\Z[\frac{1}{2}]$) induced over $\mathbf{C}$ the classical meromorphic modular function for $\Gamma(2)$: 
\begin{equation}
\lambda(z) = 16q\prod_{n \geq 1} \left(\frac{1+q^{2n}}{1+q^{2n-1}} \right)^8
\label{q_exp_lambda}
\end{equation}
where $q = e^{i\pi z}$. 
The curve $X(2)$ has three cusps, namely $[0]_{\Gamma(2)}$, $[1]_{\Gamma(2)}$ and $[\infty]_{\Gamma(2)}$. 
We have $\lambda(0) = 1$, $\lambda(i\infty)=0$ and $\lambda(1) =  \infty$ (\textit{c.f.} \cite{Pi_AGM} p. 115).

Let $c'$ be the (unique) cusp of $X(\mathbf{C})$ above $[1]_{\Gamma(2)}$ on $X(2)(\mathbf{C})$ and above $[\infty]_{\Gamma_0(p)}$ on $X_0(p)(\mathbf{C})$ (via the forgetful maps). Let $c= w_p(c')$.

The function $\lambda$ satisfies, for all $z \in \mathfrak{H}$, 
$$\lambda(z+1) = \frac{\lambda(\frac{-1}{z})-1}{\lambda(\frac{-1}{z})} $$
(\textit{c.f.} \cite{Pi_AGM} p.115). In particular, its Fourier expansion at the cusp $[1]_{\Gamma(2)}$ of $X(2)(\mathbf{C})$ begins with $\frac{-1}{16q_1}$ where $q_1 = e^{\frac{i\pi}{1-z}}$ is a local parameter at $[1]_{\Gamma(2)}$.

From now on, if $x \in \mathbf{P}^1(\mathbf{Q})$, $q_x$ will denote a local parameter at $[x]_{\Gamma_0(p) \cap \Gamma(2)}$ on $X(\mathbf{C})$.
\begin{prop}\label{uniformisantes}
Let $q_0' = e^{-\frac{2i \pi}{pz}}$ and $q_{\infty}' = e^{2 i \pi z}$ be local parameters at the cusps $[0]_{\Gamma_0(p)}$ and $[\infty]_{\Gamma_0(p)}$ respectively on $X_0(p)(\mathbf{C})$. Let $\pi , \pi':X \rightarrow X_0(p)$ given on $\mathbf{C}$ by $z \mapsto z$ and $z \mapsto \frac{z+1}{2}$ respectively.  Then we have the following relations:
\begin{itemize}
\item $q_0' \circ \pi = q_1^2$ et $q_0' \circ \pi' = q_1^4$ at $1$,
\item $q_{\infty}' \circ \pi = q_{\frac{1}{p}}^2$ et $q_{\infty}' \circ \pi' = q_{\frac{1}{p}}^4$ at $\frac{1}{p}$,
\item$q_0' \circ \pi = q_{\frac{1}{2}}^2$ et $q_0' \circ \pi' = q_{\frac{1}{2}}$ at $\frac{1}{2}$,
\item$q_{\infty}' \circ \pi = q_{\frac{2}{p}}^2$ et $q_{\infty}' \circ \pi' = q_{\frac{2}{p}}$ at $\frac{2}{p}$,
\item$q_0' \circ \pi = q_{0}^2$ et $q_0' \circ \pi' = q_{0}$ at $0$ and
\item$q_{\infty}' \circ \pi = q_{\infty}^2$ et $q_{\infty}' \circ \pi' = q_{\infty}$ at $\infty$.

Let $\psi$ be the correspondance $X \rightarrow X_0(p)$ given on complex divisors by $(z) \mapsto 2(\frac{z+1}{2}) - (z)$. Then 
$$\psi^{-1}(([\infty]_{\Gamma_0(p)})-([0]_{\Gamma_0(p)})) = 6(([\frac{1}{p}]_{\Gamma_0(p) \cap \Gamma(2)})-([1]_{\Gamma_0(p) \cap \Gamma(2)})) \text{ .}$$
\end{itemize}
\end{prop}
\begin{proof}
The relations between local parameters are direct (but tedious) computations. The reader can find an example in the proof of Proposition $1$ of \cite{Merel}. 
\end{proof}

\subsubsection{The modular unit $\mu$}\label{section_mu}
Let $\eta(z)=q^{\frac{1}{24}}\cdot \prod_{i \geq 1}(1-q^i)$ be the Dedekind eta function, where $q = e^{2 i \pi z}$, and $\Delta(z) = \eta(z)^{24}$. Let $u(z) = \left( \frac{\Delta(pz)}{\Delta(z)}\right)^{\frac{1}{d}}$; this is a rational function of $X_0(p)=M_{\Gamma_0(p)} \otimes \mathbf{Q}$ (\textit{cf.} \cite{Mazur_Eisenstein} p. 99 for the fact that we can extract the $d$-th root). 
Let $\mu$ in the function field of $X_\Q=M_{\Gamma_0(p)\cap \Gamma(2)} \otimes \Q$ be such that over $\mathbf{C}$ we have:
$$\mu(z) = \frac{u(\frac{z+1}{2})^2}{u(z)} \text{ .}$$
This is a modular unit: $\mu$ is meromorphic on $M_{\Gamma_0(p) \cap \Gamma(2)}(\mathbf{C})=X(\mathbf{C})$ with zeroes and poles at the cusps only. Since the divisor of $u$ on $X_0(p)$ is $\frac{p-1}{d} \cdot \left( ([\infty]_{\Gamma_0(p)}) - ([0]_{\Gamma_0(p)}) \right)$, Proposition \ref{uniformisantes} shows that $\mu$'s divisor is $6\cdot \frac{p-1}{d} \cdot \left( (c') - (c)\right)$. 
\begin{lemma}\label{functional_equation_mu}
We have 
$$\mu \circ w_p = p^{\frac{-12}{d}}\cdot  \mu^{-1} \text{ .}$$
\end{lemma}\
\begin{proof}
The map $X \rightarrow X_0(p)$ induced on complex points by $z \mapsto \frac{z+1}{2}$ commutes with $w_p$. Indeed, since $p>2$, it is given by $(E, C, (P,Q)) \mapsto (E/<P+Q>, (C+<P+Q>)/<P+Q>)$ in moduli terms (here $<P+Q>$ is the subgroup of $E$ generated by $P+Q$). We then readily see that it commutes with $w_p$. The lemma follows from the fact that:
$$u \circ w_p = p^{\frac{-12}{d}} \cdot u^{-1} \text{ .}$$
This is true since on complex points we have: 
$$u(\frac{-1}{pz}) = (\frac{\Delta(\frac{-1}{z})}{\Delta(\frac{-1}{pz})})^{\frac{1}{d}} = \frac{z^{\frac{12}{d}}}{(pz)^{\frac{12}{d}}} \cdot u^{-1} = p^{\frac{-12}{d}} \cdot u^{-1}(z)\text{ .}$$
\end{proof}

\begin{lemma}\label{Fourier_mu}
The first term of the Fourier expansion of $\mu$ at the cusp $c$ of $X$ is $p^{-\frac{12}{d}} \cdot q_1^{-\frac{6(p-1)}{d}}$.
\end{lemma}
\begin{proof}
Note that $c = [1]_{\Gamma_0(p) \cap \Gamma(2)}$, since $[1]_{\Gamma_0(p) \cap \Gamma(2)}$ is a pole of $\mu$ (see for example the computation below).
We have, for $z=1+h \in \mathcal{H}$:
\begin{align*}
\mu(z) &= \frac{u(1+\frac{h}{2})^2}{u(1+h)} = \frac{u(\frac{h}{2})^2}{u(h)} \\&= p^{\frac{-12}{d}} \cdot \frac{(u\circ w_p)(h)}{(u\circ w_p)(\frac{h}{2})^2} \text{ .}
\end{align*}
The assertion now follows from the fact that the expansion of $u(z)$ at the cusp $[\infty]_{\Gamma_0(p)}$ begins with $(e^{2i\pi z})^{\frac{p-1}{d}}$ (a quick computation shows that $q_1(z) = e^{\frac{i \pi}{p(1-z)}}$).
\end{proof}

\begin{lemma}\label{Fourier_lambda}
The function $\lambda \circ w_p$ of $X$ has a simple pole at $c$.
\end{lemma}
\begin{proof}
We saw that $\lambda$ has a simple pole at the cusp $[1]_{\Gamma(2)}$. A computation shows that $X$ is unramified at $c'$ over $X(2)$ (relatively to the standard projection $z \mapsto z$ of upper-half planes). Since $c'$ maps to $[1]_{\Gamma(2)}$, $\lambda$ (viewed as a function on $X$) also has a simple pole at $c'$.
\end{proof}

\subsection{Case $N=3$}\label{N=3}
\subsubsection{The cusps}
We assume $p \equiv 1 \text{ (mod }3 \text{)}$. Let $X = M_{\Gamma_0(p) \cap K}$ over $\mathbf{Z}[\frac{1}{3}]$, where $K$ is congruence subgroup of $\text{SL}_2(\mathbf{Z})$ corresponding to the diagonal matrices modulo $3$. By Remark \ref{champ N=3}, $X$ is geometrically irreducible and we have a natural isomorphism
$$X(\mathbf{C}) \simeq (\Gamma_0(p) \cap \Gamma(3)) \backslash (\mathcal{H} \cup \mathbf{P}^1 (\mathbf{Q}))  \text{ .}$$ 

The curve $X(\mathbf{C})$ has $8 = 2 \cdot \text{Card}(\mathbf{P}^1(\mathbf{F}_3))$ cusps, namely $[0]_{\Gamma_0(p) \cap \Gamma(3)}$, $[1]_{\Gamma_0(p) \cap \Gamma(3)}$, $[\frac{1}{3}]_{\Gamma_0(p) \cap \Gamma(3)}$, $[\frac{1}{2}]_{\Gamma_0(p) \cap \Gamma(3)}$, $[\infty]_{\Gamma_0(p) \cap \Gamma(3)}$, $[\frac{2}{p}]_{\Gamma_0(p) \cap \Gamma(3)}$, $[\frac{1}{p}]_{\Gamma_0(p) \cap \Gamma(3)}$, $[\frac{3}{p}]_{\Gamma_0(p) \cap \Gamma(3)}$ (the first three are above the cusp $[0]_{\Gamma_0(p)}$ of $M_{\Gamma_0(p)}(\C)$, and the last three are above $[\infty]_{\Gamma_0(p)}$; $w_p$ exchanges the first group of three cusps with the second group). The genus of $X$ is $g=p-2$, so $X$ has $g+1 = p-1$ supersingular points in its special fiber $p$.
We choose the Hauptmodul $H:M_K \xrightarrow{\sim} \mathbf{P}^1$ (defined over $\Z[\frac{1}{3}]$) such that over $\mathbf{C}$ we get the classical meromorphic function for $\Gamma(3)$: 
\begin{equation}
H(z) = \left( \frac{\eta(\frac{z}{3(1-z)})}{\eta(\frac{3z}{1-z})} \right)^3
\label{q_exp_mu3}
\end{equation}
where $\eta$ is the Dedekind eta function. 

Let $c'$ be the (unique) cusp of $X(\mathbf{C})$ above $[1]_{\Gamma(3)}$ on $X(3)(\mathbf{C})$ and above $[\infty]_{\Gamma_0(p)}$ on $X_0(p)(\mathbf{C})$ (via the forgetful maps). Let $c= w_p(c')$.

The function $H$ has a pole of order $1$ at the cusp $[1]_{\Gamma(3)}$ of $M_H$ (because a local parameter at $[1]_{\Gamma(3)}$ is $e^{\frac{-2i \pi z}{3(z-1)}}$). 
For $0 \leq i \leq 2$, let $\pi_i(z) = \frac{z+i}{3}$, $\pi_{-1}(z) = z$ and $\pi_3(z) = 3z$. Then the $\pi_i$'s induce maps $X \rightarrow X_0(p)$. We computed in table \ref{ramification} the ramification index of these maps at the cusps of $X$ above $[0]_{\Gamma_0(p)}$ (the result being the same for the cusps above $[\infty]_{\Gamma_0(p)}$, thanks to $w_p$). 

\subsubsection{The modular unit $\mu$}
We define $$\mu(z) = \frac{u(\pi_2(z))^3}{u(\pi_{-1}(z))} = \frac{u(\frac{z+2}{3})^3}{u(z)}$$
where, as in the case $N=2$, $u(z) = \left(\frac{\Delta(pz)}{\Delta(z)}\right)^{\frac{1}{d}}$ (with $d = \text{gcd}(p-1,12)$). The divisor of $\mu$ is $24\cdot \frac{p-1}{d} \cdot \left( (c') - (c)\right)$. Its degree on the various annuli is $\frac{24}{d}$. We have 
$$\mu \circ w_p = p^{\frac{-12}{d}} \cdot \mu^{-1} \text{ .}$$
Finally, $H$ (viewed as a function on $X$) has a double pole at $c'$.
\begin{table}[]
\centering
\caption{Ramification index at the cusps}
\label{ramification}
\begin{tabular}{|l|l|l|l|l|l|l|}
  \hline
  Cusp of X & Local parameter & $\pi_{-1}$ & $\pi_0$ & $\pi_1$ & $\pi_2$ & $\pi_3$\\
  \hline
  $0$ & $e^{\frac{-2i \pi}{3pz}}$ & 3 & 1 & 1 & 1 & 9 \\
  $\frac{1}{3}$ & $e^{\frac{-2i \pi}{9p(3z-1)}} $& 3 & 1 & 1 & 1 & 9 \\
  $1$ & $e^{\frac{-2i \pi}{3p(z-1)}}$ & 3 & 1 & 1 & 9 & 1\\
  $\frac{1}{2}$ &$e^{\frac{-i \pi}{3p(2z-1)}}$&3 & 1 & 9 & 1 & 1\\
  \hline
\end{tabular}
\end{table}

We define:
$$\nu(z) = \left( \prod_{0 \leq \ell \leq g} \theta(\xi_{c'}^{(\ell)},z_0,z) \right)^{\frac{24}{d}} \text{ .}$$

The same proof as in Proposition \ref{mu_nu} shows that $\mu$ and $\nu$ are proportionals. A Fourier expansion computation shows the analogue of Theorem \ref{main-thm} in the case $N=3$, with $\lambda$ remplaced by $H$.

\end{document}